\documentclass[12pt]{article}
\usepackage[utf8]{inputenc}
\usepackage[english]{babel}
\usepackage[letterpaper, margin=30mm, lmargin=30mm]{geometry}
\usepackage{amsmath}
\usepackage{amsthm}
\usepackage{amssymb}
\usepackage{graphicx}
\usepackage{tikz}
\usepackage{xy}
\usepackage[backend=bibtex,maxnames=99]{biblatex}
\usepackage{breakurl}
\usepackage{hyperref}
\usepackage{float}
\usepackage{fancyhdr}
\usetikzlibrary{positioning}
\hypersetup{
    colorlinks,
    citecolor=black,
    filecolor=black,
    linkcolor=black,
    urlcolor=black
}
\fancypagestyle{firstpage}{%

    \fancyfoot[L]{%
        \vspace*{-5\baselineskip}%
        \hrule%
        \vspace{0.5em}%
        \footnotesize This text was submitted in November, 2024 in partial fulfilment of the requirements for the %
        degree of Bachelor of Science (Honours) in Mathematics at the University of Auckland. %
        The project was~super- vised by Professor Andr\'e Nies and Associate Professor Sina Greenwood.}
    \fancyfoot[C]{\vspace{-1\baselineskip}\thepage}
}

\newtheorem{theorem}{Theorem}[section]
\newtheorem{proposition}[theorem]{Proposition}
\newtheorem{lemma}[theorem]{Lemma}
\newtheorem{definition}[theorem]{Definition}

\newtheorem{claim}[theorem]{Claim}

\numberwithin{equation}{section}

\newcommand{\hatc}{\widehat{\,\,\,}}
\newcommand{\wtow}{{}^\omega\omega}
\newcommand{\infsets}{[\omega]^\omega}

\newcommand{\card}{\operatorname{Card}}
\newcommand{\high}{\operatorname{H}}
\newcommand{\nonlow}{\operatorname{NL}}
\newcommand{\smm}{\mathfrak s_{\mathrm{mm}}}

\begin{document}

\title{Cardinal Characteristics and Computability}
\date{}
\author{Logan McDonald}

\maketitle
\thispagestyle{firstpage}

\begin{abstract}
    Cardinal characteristics of the continuum represent the boundaries in size between the countable and the continuum with respect to certain properties of sets.
    They are often defined as the minimum sizes of families of reals that meet some criteria.
    Taking these families and considering their analogues in the setting of computability theory provides a rich hierarchy of properties of oracles, which can be studied in terms of the Muchnik/Medvedev lattices of mass problems.
    We provide more detail to the proof of the Medvedev equivalence between dominating functions and maximal independent families given by Lempp et al.\ (2023) and adapt their construction of maximal almost disjoint families to the setting of $\omega$-computably approximable sets. 
    We then extend the theory to include correspondents of maximal ideal independent families and show they behave similarly to the maximal independent families.
\end{abstract}

\begin{figure}[b]
    \vspace*{25pt}
\end{figure}

\tableofcontents

\section{Introduction}
The continuum is an essential concept in many areas of mathematics.
Among the most famous results in mathematical logic are the two statements that comprise the undecidability of the cardinality of the continuum.
The continuum hypothesis (CH), as posed by Cantor in 1878, states: 
``There is no set whose cardinality is strictly between that of the integers and that of the reals."
Work by G\"odel \cite{God:1939} in 1940 and Cohen~\cite{Coh:1963} in 1963 shows that this is undecidable in the framework of Zermelo-Frankael set theory (ZFC), a sufficient foundation for much of mathematics.
If we assume ZFC is consistent, then there are models of ZFC in which CH holds and models where it does not hold.
\par
Despite CH being independent of ZFC, there is a good amount of thought to be given to the size of the continuum.
ZFC is very far from being an absolute characterisation of mathematics, so the undecidability of the size of the continuum does not ensure that there is no evidence for or against CH.
When working with mathematics, one can accept some rigorous yet incomplete view of the universe of sets, such as ZFC, or hold a Platonistic view and assert that there is some underlying structure to mathematics that we can only approximate with our formal theories.
In the latter case, one may believe that the size of the continuum will have an evidenced and intuitive answer, but we have not found it.
If this were the case, it may seem strange we still use an axiomatic system that has not changed for over a century, as that would suggest no progress has been made in this direction.
\par
However, there are axioms we could add, the axiom of projective determinacy (PD) is one such case.
Projective determinacy is believed to be true by Woodin \cite{Woo:2001}, and like CH it is independent from ZFC.
Regarding axioms such as PD, Woodin \cite[p.\ 569]{Woo:2001} says:
\begin{quote}
    The truth of these axioms became evident only after a great deal of work. 
    For me, a remarkable aspect of this is that it demonstrates that the discovery of mathematical truth is not a purely formal endeavor.
\end{quote}
\par
Several attempts have been made to resolve the continuum problem in the way some think of projective determinacy.
Woodin \cite{Woo:2004} developed his notion of $\Omega$-logic as one approach to finding evidence against CH.
His work in this direction eventually all depended on a very strong conjecture.
Perhaps because of this, in 2010, a shift occured in Woodin's thoughts, with him beginning work in a new direction.
We take a slight detour to discuss this.
\par
G\"odel's proof that CH is consistent used his constructible universe $L$, which satisfies the axioms of ZFC and CH.
The assumption that the set theoretic universe $V$ is equal to $L$ could be taken, and indeed, it resolves many undecidable statements in ZFC.
Unfortunately, ZFC+``$V=L$" is very restrictive.
Not only does it satisfy CH, it also contradicts many large cardinal axioms that speak about the size of the set theoretic universe.
Woodin \cite{Woo:2017} believes that there is no reason to rule out these large cardinal axioms, and so now works on developing an alternative to $V=L$ which accommodates these large cardinal axioms while preserving CH.
He looks for an ultimate version of~$L$, which should be compatible with many large cardinals, suggesting that it is always close to $V$; then arguably the analogue of $V=L$ may be true.
\par
Cardinal characteristics of the continuum are sizes of families of reals that preserve some properties of the reals.
They act as a sort of boundary between the countable sets and the reals with respect to some property.
The range of possibilities for cardinal characteristics depends on CH.
If CH holds, the only uncountable possibility is $\mathfrak c$, the cardinality of the continuum.
If we do not assume CH, perhaps we can find a richer structure below the continuum; these cardinal characteristics provide an interesting direction to study. 
\par
As CH trivialises the theory of cardinal characteristics of the continuum, the rich structure of cardinal characteristics could be taken as evidence against CH.
There is no apparent reason that all of these values should coincide, as CH suggests.
\par
The main topic of interest in this dissertation is the analogues of these cardinal characteristics in the setting of computability theory, mass problems in particular.
We can study these in terms of lattices of degrees under the Muchnik and Medvedev reductions.
Much the same as Turing reduction, which gives a degree structure and notion of computational complexity on sets of natural numbers, the Muchnik and Medvedev reductions give a similar structure on collections of functions or sets on the natural numbers.
\par
The lattices of Muchnik and Medvedev degrees have notably been studied at length by Simpson \cite{CS:2007,Sim:2008,Sim:2005}.
It is a popular opinion that the structure of $\Pi^0_1$-classes with the Muchnik and Medvedev reducibilities is natural to study, relative to the c.e. degrees of sets, as noted by Binns \cite{Bin:2002}.
We can see an example of this with how easily computability theoretic correspondents arise from the cardinal characteristics studied in set theory.

\subsection{Notation}
We elaborate on some of the set theoretic notation that will be used throughout this dissertation.
$\wtow$ is the set of functions from naturals to naturals (to avoid being confused with the ordinal $\omega^\omega$).
Following suit, ${}^\omega X$ is used for countably infinite sequences of elements from $X$.
The set of infinite subsets of the naturals is written $\infsets$.
The quantifiers $\forall^\infty$ and $\exists^\infty$ are used to mean `for all but finitely many' and `there exist infinitely many' respectively.
Note that `for all but finitely many' is equivalent to excluding some finite initial segment.
These follow the same duality relation as the usual quantifiers $\forall$ and $\exists$ via negation.
An asterisk is often used to represent mod finite relations.
For example, $A\subseteq^*B$ means that $A$ is a subset of $B$ if you ignore finitely many exceptions.
The word `almost' is also often used to mean with finitely many exceptions.
Given a set $A$, its characteristic function $\chi_A$ is defined to be $1$ on inputs within $A$ and $0$ everywhere else.

\subsection{Cardinal characteristics of the continuum} \label{ccs}
Among the simplest cardinal characteristics are the bounding number $\mathfrak b$ and the dominating number $\mathfrak d$.
We look at the mod finite relation between functions on the natural numbers.
For $f,g\in\wtow$ we have $f\le^*g$ iff $\forall^\infty x:f(x)\le g(x)$, there is some fixed $x_0$ such that $f$ is at least as much as $g$ past that point.
One can define the reverse and non-strict relations in much the same way.
\par
The \emph{bounding number} $\mathfrak b$ is the least size of an unbounded family in $\wtow$ ordered by~$\le^*$.
Specifically a family $\mathcal B\subseteq\wtow$ is \emph{unbounded} if there is no choice of a single $f\in\wtow$ such that $g\le^*f$ for all $g\in\mathcal B$.
The \emph{dominating number} $\mathfrak d$ is similarly the least size of a dominating family.
A family $\mathcal D\subseteq\wtow$ is \emph{dominating} if for each $f\in\wtow$ there is a $g\in\mathcal D$ such that $f\le^*g$.
Clearly $\mathfrak b\le\mathfrak d$ (a family cannot be dominating if it is not unbounded).
An unbounded family must be uncountable, this can be observed via a diagonalisation argument.
Suppose we have a countable family $\{g_n:n\in\omega\}$, then define $f(n):=\max_{i\le n}g_i(n)$ which ensures the family is not unbounded.
\par
The bounding number was first considered by Rothberger \cite{Rot:1938}, who studied its relation to the continuum hypothesis.
The dominating number was introduced by Kat\v etov \cite{Kat:1960}.
Some more classical results in cardinal characteristics are described and referenced in van Douwen \cite{Dou:1984}.
These cardinals have had various notations. 
The first notation for $\mathfrak b$ was $\aleph_\eta$, and the symbols $K_8$, $\lambda_3$, and $\xi$ have also been used according to van Douwen \cite{Dou:1984}.
\par
Several further characteristics are defined in terms of families of sets of naturals, rather than functions on the natural numbers.
A family ${\mathcal A\subseteq\infsets}$ is called \emph{almost disjoint} if the intersection of any pair of elements is finite.
An almost disjoint family is maximal if no more sets can be added while preserving almost disjointness.
In general a set is maximal with respect to a property if it can not be extended while preserving the property.
The \emph{almost disjointness number} $\mathfrak a$ is the least size of a maximal almost disjoint (MAD) family.
\par
A family $\mathcal T\subseteq\infsets$ is called a \emph{tower} if it is linearly ordered by $\subseteq^*$.
The \emph{tower number} $\mathfrak t$ is the least size of a maximal tower.
In a weakening of this, a family has the \emph{strong finite intersection property} if any finite subset has infinite intersection.
Such a family is said to have a \emph{pseudointersection} if there is an infinite set which is an almost subset of every element.
The \emph{pseudointersection number} $\mathfrak p$ is the least size of a family which has the strong finite intersection property but no pseudointersection.
This is a weaker condition than being a maximal tower so $\mathfrak p\le\mathfrak t$.
It was shown by Malliaris and Shelah~\cite{MS:2016} that $\mathfrak p=\mathfrak t$ is in fact provable in ZFC, a rare case of a new ZFC relation between characteristics.
\par
Recall that a filter base is a family of sets whose upwards closure forms a filter.
The \emph{ultrafilter number} $\mathfrak u$ is the least size of a filter base for a non-principal ultrafilter on~$\omega$.
A family $\mathcal I\subseteq\infsets$ is called \emph{independent} if any finite intersection of its elements or their complements is infinite.
The \emph{independence number} $\mathfrak i$ is the least size of a maximal independent family.
\par
Blass \cite{Bla:2010} gives a survey of the theory of cardinal characteristics.
It provides in-depth information about these and many other characteristics.
In Figure \ref{hasseccc}, adapted from Blass' survey and adding the result of Malliaris and Shelah, we can see which relations between these and other combinatorial cardinal characteristics are known to be provable in ZFC.
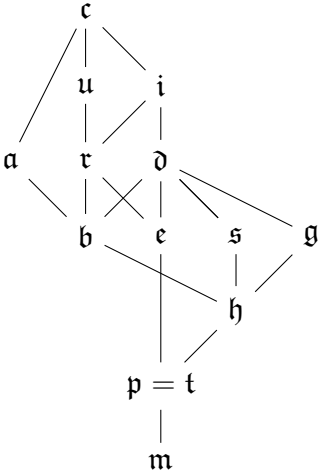
\begin{figure}[H]
\begin{tikzpicture}
    \node (u) {$\mathfrak u$};
    \node (c) [above of = u] {$\mathfrak c$};
    \node (i) [right of = u] {$\mathfrak i$};
    \node (r) [below of = u] {$\mathfrak r$};
    \node (d) [below of = i] {$\mathfrak d$};
    \node (a) [left of = r] {$\mathfrak a$};
    \node (b) [below of = r] {$\mathfrak b$};
    \node (e) [below of = d] {$\mathfrak e$};
    \node (s) [right of = e] {$\mathfrak s$};
    \node (g) [right of = s] {$\mathfrak g$};
    \node (h) [below of = s] {$\mathfrak h$};
    \node (blank) [below of = e] {};
    \node (t) [below of = blank] {$\mathfrak {p=t}$};
    \node (m) [below of = t] {$\mathfrak m$};

    \draw (u) -- (r);\draw (c) -- (u);\draw (c) -- (a);
    \draw (c) -- (i);\draw (i) -- (r);\draw (i) -- (d);
    \draw (a) -- (b);\draw (r) -- (b);\draw (d) -- (b);
    \draw (r) -- (e);\draw (d) -- (e);\draw (d) -- (g);
    \draw (d) -- (s);\draw (d) -- (s);\draw (b) -- (h);
    \draw (g) -- (h);\draw (s) -- (h);\draw (h) -- (t);
    \draw (e) -- (t);\draw (t) -- (m);
\end{tikzpicture}
\centering
\caption{Hasse diagram of many combinatorial cardinal characteristics}
\label{hasseccc}
\end{figure}
The result that $\mathfrak p=\mathfrak t$ is a truly rare case of an equality between characteristics being provable in ZFC.
There is typically very little known evidence that any two cardinal characteristics should (not necessarily in a formal sense) be equal, unless one believes that CH is true.
The rich structure of cardinal characteristics and the difficulty of obtaining equality between them could suggest that CH is not true in an intuitive sense.
Perhaps the structure below the continuum has a natural significance.


\subsection{Preliminaries on families in Boolean algebras} \label{bprelims}
The analogues of cardinal characteristics that we focus on are \emph{mass problems}, sets of functions on $\omega$.
Note that families of sets are a special case of this.
We can compare relative computational complexity of these problems via the \emph{Medvedev (strong) reduction}: $\mathcal A\le_s\mathcal B$ for mass problems $\mathcal A,\mathcal B$ if there is a fixed Turing functional $\Theta$ such that $\Theta^B\in\mathcal A$ for any $B\in\mathcal B$.
That is to say that there is a uniform reduction that takes elements of $\mathcal B$ to elements of $\mathcal A$.
Alternatively we can use the \emph{Muchnik (weak) reduction} where a fixed Turing function is not necessary, simply $\mathcal A\le_w\mathcal B$ if every element of~$\mathcal B$ computes an element of $\mathcal A$.
Several cardinal characteristics can be associated with mass problems, and we can observe that there is some correlation between the ordering of cardinal characteristics and the position of the mass problem in the Medvedev or Muchnik hierarchies.
\par
In order to get mass problems from the cardinal characteristics defined in terms of families of sets we use some definitions.
We fix a Boolean algebra $\mathbb B$ of subsets of $\omega$, and look at subsets of $\mathbb B$ which will be our mass problems.
When the Boolean algebra~$\mathbb B$ is not specified it is assumed to be the computable sets.
It is important to be able to view sequences of sets as subsets of $\omega$ themselves.
Given a suitable pairing function $\langle\cdot,\cdot\rangle:\omega\times\omega\to\omega$ (it should be bijective and always have $\langle x,e\rangle\ge x,e$) we can view a sequence $\langle F_e\rangle_{e\in\omega}$ as the set $F$ precisely containing elements $\langle x,e\rangle$ where $x\in F_e$.
Something similar can be done for functions, a sequence $\langle f_e\rangle_{e\in\omega}$ is identified with a function $f$ where $f(e,x)=f_e(x)$.
\par
Now we look at analogies of several of the combinatorial cardinal characteristics defined earlier.
\begin{definition} \label{defmad}
A sequence $\langle F_e\rangle_{e\in\omega}$ is called a maximal almost disjoint (MAD) family in $\mathbb B$ if for any $e,k$ we have $F_e\cap F_k$ finite and for any infinite $R\in\mathbb B$ there is some $e$ with $F_e\cap R$ infinite.
Let $\mathcal A_{\mathbb B}$ be the mass problem of MAD families in $\mathbb B$.
\end{definition}
\begin{definition} \label{deftower}
A $\mathbb B$-tower is defined to be a sequence $\langle F_e\rangle_{e\in\omega}$ such that $F_0=\omega$, $F_{e+1}\subseteq^*F_e$, and $F_e\smallsetminus F_{e+1}$ infinite for each $e$. 
It is maximal if for every infinite $R\in\mathbb B$, there is $e$ such that $R\smallsetminus F_e$ is infinite.
$\mathcal T_{\mathbb B}$ is the mass problem of maximal towers in $\mathbb B$.
\end{definition}
Lempp et al.\ show that $\mathcal A_{\mathbb B}\le_s\mathcal T_{\mathbb B}\le_s\mathcal A_{\mathbb B}$ for any Boolean algebra $\mathbb B$ of subsets of $\omega$.
This equivalence does not carry over to the set theoretic setting, where we have only that $\mathfrak a\ge\mathfrak t$.
\par
\begin{definition} \label{defufb}
A $\mathbb B$-ultrafilter base is defined to be a $\mathbb B$-tower $\langle F_e\rangle_{e\in\omega}$ such that for every $R\in\mathbb B$ there is an $e$ such that either $F_e\subseteq^*\overline R$ or $F_e\subseteq^* R$.
Let $\mathcal U_{\mathbb B}$ be the mass problem of $\mathbb B$-ultrafilter bases.
\end{definition}
\begin{definition} \label{defidp}
Given a sequence $\langle F_e\rangle_{e\in\omega}$, for each binary string $\sigma$ define 
\begin{equation}
F_\sigma=\bigcap_{\sigma(i)=0}\overline F_i\cap\bigcap_{\sigma(i)=1}F_i.
\label{fsigma}
\end{equation}
The sequence is called $\mathbb B$-independent if for every binary string $\sigma$, the set $F_\sigma$ is infinite.
The sequence is maximal independent in $\mathbb B$ if for every $R\in\mathbb B$ there is $\sigma$ such that $F_\sigma\subseteq^*R$ or $F_\sigma\subseteq^*\overline R$.
Let $\mathcal I_{\mathbb B}$ be the mass problem of maximal independent families in~$\mathbb B$.
\end{definition}
Lempp et al.\ show that when $\mathbb B$ is the Boolean algebra of computable sets modulo finite sets then $\mathcal U_{\mathbb B}=_s\mathcal I_{\mathbb B}$.
This result is obtained by showing that both are Medvedev equivalent to the mass problem of dominating functions for computable sets.
\par
A function $g:\omega\to\omega$ is called dominating if $g\ge^*f$ for every computable function~$f$ and $g(n)\ge n$ for all $n$.
The mass problem of dominating functions is called $\mathrm{DomFcn}$.

\subsection{Recursive analogues of characteristics}
\
Rupprecht in his PhD thesis \cite{Rup:2010-1} gave an exposition on what he called `Turing characteristics'.
He defines a scheme for generalising the cardinal characteristics' definitions, associating them to objects that he calls `debates'.
We will look more closely at these objects later in this introduction.
In particular, Rupprecht looks at the cardinals in Cicho\'n's diagram, which shows relationships between $\mathfrak b$, $\mathfrak d$, and characteristics related to measure and Baire category.
Using debates, he studies the analogues of these cardinals in the computability theory setting.
Most of his work is only in his thesis, and otherwise unpublished.
\par
Further work on this was done by Brendle et al.\ \cite{BBNN:2015}.
This paper made the theory more accessible and expanded on several interesting directions.
They survey the theory regarding Cicho\'n's diagram and relate the theory to other notions in both computability and set theory.
They also outline further directions for research: open problems and set theoretic relations that were not known to hold or not in computability.
\par
The theory of cardinal characteristics in computability is not isolated; it has applications to other areas of computability theory.
The Gamma question comes from the field of coarse computability, and was resolved by Monin \cite{Mon:2018} using the theory of mass problems.
It is an interesting example to show how uses of the theory of mass problems arise naturally, as will be seen in section \ref{gammaprob}.
\subsubsection{Highness classes and Weihrauch problems}
Greenberg, Kuyper, and Turetsky \cite{GKT:2019} studied how Weihrauch problems are used to find analogous results in set theory and computability.
These Weihrauch problems are the debates that Rupprecht had studied, given the name by which they are known in reverse mathematics.
A \emph{Weihrauch problem} is a triple $A=(A_{\mathrm{inst}},A_{\mathrm{sol}},A)$ with ${A_{\mathrm{inst}},A_{\mathrm{sol}}\subseteq\wtow}$ and $A$ is a binary relation between $A_{\mathrm{inst}}$ and $A_{\mathrm{sol}}$.
Elements of $A_{\mathrm{inst}}$ and $A_{\mathrm{sol}}$ are referred to as \emph{instances} and \emph{solutions} respectively.
$b\in A_{\mathrm{sol}}$ is a \emph{solution for $a\in A_{\mathrm{inst}}$} if $a A b$.
For example, the domination problem $\mathrm{Dom}$ is related to the characteristic $\mathfrak d$.
An instance is a function $f\in\wtow$ and a solution for it is a $g\in\wtow$ such that $g\ge^*f$.
A \emph{complete solution set} for $A$ is a set of solutions containing a solution for every instance of $A$.
$\card(A)=\min\{|C|:C\text{ is a complete solution set for }A\}$.
$\high(A)$ is the collection of oracles $X\in{}^\omega2$ that compute a solution which solves every computable instance of $A$.
$\card(A)$ gives us a cardinal characteristics associated with $A$, and $\high(A)$ a highness class associated with $A$.
As one would expect, $\card(\mathrm{Dom})=\mathfrak d$.
For $\mathrm{Dom}$, we get $\high(\mathrm{Dom})$ to be the class of high degrees, which compute functions that dominate all computable functions.
\par
This theory of Weihrauch problems gives us a scheme that highlights the analogy between these cardinal characteristics and interesting notions in computability.
Rupprecht and Brendle et al.\ use different terminology and notation for these concepts, though the ideas remain the same.
They are described here in this form, as by Greenberg, Kuyper, and Turetsky \cite{GKT:2019}, as it is the most recent and tidy description and best integrates with similar structures in other areas of computability and logic.
\par
Greenberg, Kuyper, and Turetsky also observe that the relations between cardinal characteristics provable in ZFC or reductions between highness properties can be described in this setting.
An \emph{effective morphism} between two Weihrauch problems $A,B$ is a reduction from $A$ to $B$ in this setting.
Given an instance $a$ for $A$, it is effectively translated to an instance for $B$, and $B$-solutions for this instance are translated back to $A$-solutions for $a$.
The effective morphism is thus a pair of maps $\varphi_{\mathrm{inst}},\varphi_{\mathrm{sol}}$ such that $\varphi_{\mathrm{inst}}(A_{\mathrm{inst}})\subseteq B_{\mathrm{inst}}$, $\varphi_{\mathrm{sol}}(B_{\mathrm{sol}})\subseteq A_{\mathrm{sol}}$, and $\varphi_{\mathrm{inst}}(a)Bb$ implies $aA\varphi_{\mathrm{sol}}(b)$.
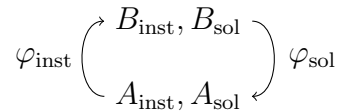
\begin{figure}[h]
\begin{tikzpicture}
    \node(b) {$B_{\mathrm{inst}},B_{\mathrm{sol}}$};
    \node(a) [below of = b] {$A_{\mathrm{inst}},A_{\mathrm{sol}}$};

    \draw (a.west) edge[->] [out=180,in=180] (b.west) node[xshift=-0.8cm,yshift=0.5cm] {$\varphi_{\mathrm{inst}}$};
    \draw (b.east) edge[->] [out=0,in=0] (a.east) node[xshift=0.8cm,yshift=-0.5cm] {$\varphi_{\mathrm{sol}}$};
\end{tikzpicture}
\centering
\label{effmorph}
\caption{An effective morphism}
\end{figure}
The existence of an effective morphism from $A$ to $B$ is written $A\to B$.
Greenberg, Kuyper, and Turetsky show that $A\to B$ implies that $\card(A)\le\card(B)$ and $\high(A)\supseteq\high(B)$, justifying this scheme as a sufficient framework to study these analogies.
\par
It is interesting to note that in the cases of dual characteristics such as $\mathfrak b$ and $\mathfrak d$, there are two ways that analogues can be described via the above methods.
Whereas Greenberg, Kuyper, and Turetsky \cite{GKT:2019} take high degrees as equivalent to $\mathfrak d$ and hyperimmune degrees for $\mathfrak b$.
Brendle et al.\ \cite{BBNN:2015} do the reverse, $\mathfrak b$ corresponding to high degrees and $\mathfrak d$ to hyperimmune degrees.
Both papers discuss this in some form.
Greenberg, Kuyper, and Turetsky define both the highness class $\high(A)$ and the non-lowness class~$\nonlow(A)$ and give similar results as above for the non-lowness class $\nonlow(A)$ but with the direction flipped.
They note that the non-lowness classes behave better in some ways and use them for the majority of the paper.
Brendle et al.\ justify their choice via a comparison between building an oracle and building a generic object for forcing.
The preferred method in this dissertation is to associate $\mathfrak d$ with dominating functions and high degrees and $\mathfrak b$ with hyperimmune degrees.
This best aligns with the analogues of many of the characteristics in Figure \ref{hasseccc}.
\subsubsection{Mass problems in the Medvedev lattice}
While the highness classes in Turing degrees provide one system of analogues to cardinal characteristics, the theory of mass problems provides an alternative.
Whereas $\high(A)$ takes the class of oracles that compute a universal solution for computable instances of $A$, we could instead take the universal solutions themselves:
\begin{definition} \label{bdcomp}
    For a relation $R$ on $X,Y$ define
    $$\mathcal B(R)=\{y\in Y:\forall x\text{ computable }(xRy)\}$$ 
    $$\mathcal D(R)=\{x\in X:\forall y\text{ computable }(\lnot xRy)\}$$
\end{definition}
Here, we follow the convention of Brendle et al.\ as it will be convenient to refer to them later.
In the case of the eventual domination relation, the former gives us the mass problem $\mathcal B(\le^*)=\mathrm{DomFcn}$ of dominating functions, which is used at length in the later chapters.
We could also take the mass problem of the high sets themselves, which is equivalent to $\mathrm{DomFcn}$ in the Muchnik/Medvedev degrees, but includes elements higher in the Turing degrees that have nothing to do with $\mathrm{Dom}$ other than computing a dominating function.
This further suggests the Muchnik degrees could be a more natural class of objects to study than the Turing degrees for these objects.
\par
Instead of a mass problem equivalent of highness classes, it can be interesting to define an analogue of the set theoretic complete solution set defined above.
Using a sequence identified with a single object (as described above), the analogue of a complete solution set for $A$ will be a sequence $\langle a_e\rangle_{e\in\omega}$ of computable elements that contains a solution for any computable instance of $A$.
The set of all such sequences is another mass problem associated with $A$.
If we take the example of $\mathrm{Dom}$ as above, we get the mass problem of sequences $\langle g_e\rangle_{e\in\omega}$ such that for any computable function $f$ there is some $e$ such that $f\le ^*g_e$.
This mass problem is equivalent in the Muchnik/Medvedev degrees to the mass problems of high sets or dominating functions.
\par
This gives an idea of how to find analogues of cardinal characteristics that do not fit this framework.
Many of the characteristics mentioned in Section \ref{ccs} can be converted into computability theoretic objects by taking the families they bound the size of and converting them directly into families in computability.
This gives us the notions described in Section \ref{bprelims}, and the associated mass problems give us a natural way of studying their complexity.
These are studied at length by Lempp et al.\ \cite{LMNS:2023}.
Notably, they show that the mass problems $\mathcal A$ and $\mathcal T$ of MAD families and maximal towers are Medvedev equivalent, and similarly that $\mathcal U$ for ultrafilter bases and $\mathcal I$ for independent sets are Medvedev equivalent.
\subsubsection{Other work on cardinal characteristics in computability}
There are many other cardinal characteristics that can be looked at in the setting of computability.
The final comments of Brendle et al.\ \cite{BBNN:2015} address some results regarding analogues of the splitting number $\mathfrak s$ and the reaping number $\mathfrak r$.
Valverde and Tveite \cite{VT:2021} investigate correspondents of the evasion number $\mathfrak e$ and the predicting number $\mathfrak v$.
We can see in Figure \ref{otherchars} how the set theoretic and computability theoretic relations compare.
Figure \ref{otherchars} follows the conventions of Brendle et al.\ to associate $\mathfrak b$ to high degrees as $\mathcal B(\le^*)$ and $\mathfrak d$ to hyperimmune degrees as $\mathcal D(\le^*)$ in Definition \ref{bdcomp}.
The depicted set theoretic results are given by Blass \cite{Bla:2010}.
\begin{figure}[h]
\begin{minipage}{.3\textwidth}
\begin{tikzpicture}
    \node(d) {$\mathfrak d$};
    \node(s) [below = of d] {$\mathfrak s$};
    \node(r) [right = of d] {$\mathfrak r$};
    \node(e) [left = of d] {$\mathfrak e$};
    \node(v) [above = of d] {$\mathfrak v$};
    \node(b) [below = of e] {$\mathfrak b$};

    \draw[->] (b) -- (d);\draw[->] (e) -- (d);\draw[->] (s) -- (d);
    \draw[->] (e) -- (v);\draw[->] (s) -- (r);\draw[->] (b) -- (r);
    \draw[->] (b) -- (v);\draw[->] (e) to[out=335, in=205] (r);
\end{tikzpicture}
\end{minipage}
\begin{minipage}{.6\textwidth}
\scalebox{0.8}{
\begin{tikzpicture}
    \node(dc) [text width=2cm, align=center] {Hyper-immune degree};
    \node(sc) [text width=2cm, align=center, below = of dc] {r-cohesive};
    \node(rc) [text width=2cm, align=center, right = 2cm of dc] {Bi-immune degree};
    \node(ec) [text width=2cm, align=center, left = 2cm of dc] {Prediction degree};
    \node(vc) [text width=2cm, align=center, above = of dc] {Evasion degree};
    \node(bc) [text width=2cm, align=center, below = of ec] {High degree};

    \draw[->] (bc) -- (ec) node[xshift=-1.2cm, yshift=-1.1cm] {\cite[Th.\ 3.1]{VT:2021}};
    \draw[->] (bc) -- (sc) node[xshift=-2.15cm, yshift=0.4cm] {\cite[Cor.\ 2]{Joc:1973}};
    \draw[->] (ec) -- (dc) node[xshift=-2.15cm, yshift=0.4cm] {\cite[Th.\ 3.4]{VT:2021}};
    \draw[->] (sc) -- (dc) node[xshift=1.3cm, yshift=-1.4cm] {\cite[Lem.\ 2.7]{GKT:2019}};
    \draw[->] (dc) -- (vc) node[xshift=-1.2cm, yshift=-1.1cm] {\cite[Th.\ 4.1]{VT:2021}};
    \draw[->] (dc) -- (rc) node[xshift=-2.15cm, yshift=0.4cm] {\cite[Th.\ 3]{Joc:1969}};
\end{tikzpicture}}
\end{minipage}
\centering
\caption{Relations between some cardinal characteristics and their analogues}
\label{otherchars}
\end{figure}
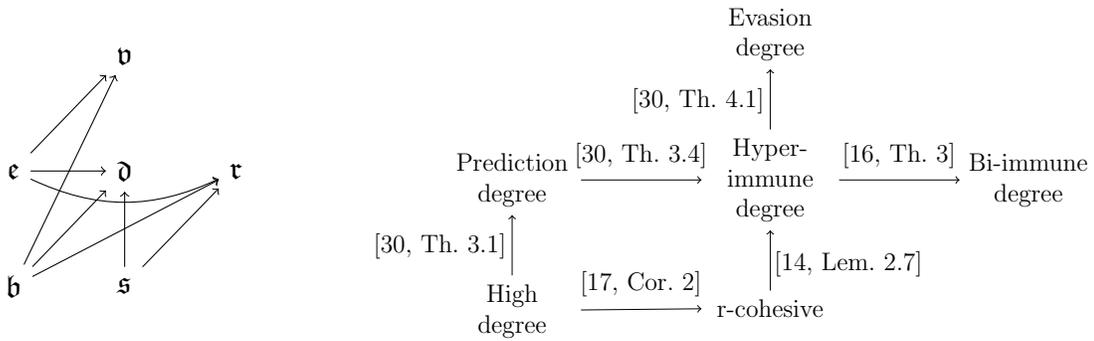
We can see here some instances of being able to obtain more in the computability setting, as shown in Figure \ref{otherchars}.
Several of the results cited in the diagram are due to equivalences in the computability setting.
For example, the hyperimmune degrees and weakly 1-generic degrees coincide, so to prove a property implies hyperimmune degree it is sometimes more instructive to show that it has weakly 1-generic degree.
This is how prediction degree implies hyperimmune degree was shown by Valverde and Tveite \cite{VT:2021}.
\par
It is an important aspect of the cardinal characteristics program to see where it is possible to separate the cardinal characteristics in ZFC. 
Naturally, we want to see how this carries over into the computability theory case, but this is where more differences between set theory and computability arise.
It is the case that several analogues of cardinal characteristics turn out to be equivalent, where strict relations are consistent in set theory.
Several of the analogues of Cicho\'n's diagram collapse in the computability theory setting, as noted by Brendle et al.\ \cite{BBNN:2015}.
Perhaps this indicates that other analogues may be more fruitful; Kihara \cite{Kih:2017} has shown that in the setting of hyperarithmetical theory less collapsing occurs in analogues of Cicho\'n's diagram.
Significant work has been done by Switzer \cite{Swi:2020-1,Swi:2020-2} describing analogues of Cicho\'n's diagram for generalised reduction concepts, in particular degrees of constructability relative to a fixed inner model of ZFC.
He also considers relationships between cardinal characteristics of the space of functions $\wtow\to\wtow$.

\subsubsection{The Gamma question and mass problems} \label{gammaprob}
The Gamma question asks about the Gamma value of Turing degrees, which in some sense measures how close a degree is to being computable.
The \emph{lower density} of $Z\subseteq\omega$ is defined to be the number $$\underline\rho(Z)=\liminf_n\frac{|Z\cap[0,n]|}n$$ between $0$ and $1$.
Given sequences $A$ and $R$, let $A\leftrightarrow R$ be the set of indices $n$ where the $n$-th element of $A$ is equal to the $n$-th element of $R$.
The \emph{gamma value} of a sequence $A$ is defined to be: $$\gamma(A)=\sup_{R\text{ computable}}\underline\rho(A\leftrightarrow R)$$
Generalising this to Turing degrees, the \emph{Gamma value} (with a capital `G') of a degree with representative $X$ is: $$\Gamma(X)=\inf\{\gamma(A):A\equiv_TX\}$$
The Gamma question asks if the Gamma value can be strictly between $0$ and $1/2$, if not then the only possible values are in $\{0,1/2,1\}$.
Building on work by Rupprecht \cite{Rup:2010-2} on Schnorr engulfing sequences and hyperimmune degrees, Monin and Nies \cite{MN:2015} made progress on the Gamma question.
They give conditions on the Gamma value of oracles that lie in the families studied by Rupprecht. 
Work by Monin \cite{Mon:2018} builds on this and gives a negative answer to the Gamma question.
There was more progress by Monin and Nies \cite{MN:2021} which among other things gave a refined proof of this result. 
\par
The proof makes use of mass problems that had already been studied, and introduces mass problems directly related to the Gamma question.
A function $f$ is called IOE (``infinitely often equal") if for any computable function $r$ we have $\exists^\infty x(f(x)=r(x))$, that is $f$ agrees with each computable function infinitely often.
Dually, a function~$f$ is called AED (``almost everywhere different") if $\forall^\infty x(f(x)\neq r(x))$.
Given an order function $h$ (non-decreasing, unbounded, computable) we define classes $\mathrm{IOE}(h)$ and $\mathrm{AED}(h)$ as above, but with the constraint that $r<h$.
To fit with Definition \ref{bdcomp}, one can define a relation $f\neq_h^*g\iff f\neq^*g\,\,\land\,\,g<h$ that gives $\mathcal D(\neq_h^*)=\mathrm{IOE}(h)$ and $\mathcal B(\neq_h^*)=\mathrm{AED}(h)$.
\par
In order to tie these previously studied objects to the Gamma question, note how the gamma value as a supremum of the lower density over all computable sequences is reminiscent of the universal solutions we use to obtain mass problems.
To bring these together concretely, the relation $x\bowtie_py\iff\underline\rho(x\leftrightarrow y)>p$ is defined for $p\in[0,1)$.
The mass problems $\mathcal B(\bowtie_p)$ and $\mathcal D(\bowtie_p)$ are abbreviated as $\mathcal B(p)$ and $\mathcal D(p)$ respectively.
The main result they seek to prove is that
$$\mathcal D(p)\equiv_w\mathrm{IOE}(2^{2^n})\text{ and }\mathcal B(p)\equiv_s\mathrm{AED}(2^{2^n})$$
for arbitrary $p\in(0,1/2)$.
Once they have this, they have that the degrees of sequences with gamma value below $1/2$ also have representatives with gamma value below any $p>0$, so the infimum and hence Gamma value must be $0$.

\section{Combinatorial Mass Problems}
The cardinal characteristics with definitions based on families of sets with self-contained properties did not fit into the frameworks based on relations that were described earlier.
The work on finding correspondents to these done by Lempp et al.\ \cite{LMNS:2023} shows that many of the relations between these cardinal characteristics collapse in the setting of mass problems.
\par
Lempp et al.\ show that in any Boolean algebra of sets, the mass problems of MAD families and of maximal towers are Medvedev equivalent.
In set theory, it is consistent for $\mathfrak a>\mathfrak t$ to hold, so this is one instance of collapsing in the computability setting.
They also show that any non-computable c.e.\ set computes a MAD family and that this MAD family is even c.e.
The Medvedev equivalence gives a co-c.e.\ maximal tower from this.
\par
Specifically for the computable sets, by showing both are equivalent to the mass problem of dominating functions, Lempp et al.\ show that $\mathcal U$ and $\mathcal I$ are Medvedev equivalent.
No pair of direct reductions between the two is known.
The equivalence to dominating functions shows that ultrafilter bases (and maximal independent families) capture highness, that computing an ultrafilter base is just as good as computing a dominating function.
They also construct a co-c.e.\ ultrafilter base for the computable sets to give an example that is quite simple in a descriptive sense.
\par
The reductions that give the equivalence $\mathcal I\equiv_s\mathrm{DomFcn}$ are described below with more detail to add clarity and to compare with the results of Section \ref{iif}.
In addition, we adapt their construction of a co-c.e.\ MAD family to the setting of $\omega$-computably approximable sets.

\subsection{Independent families and dominating functions}
Define $\mathrm{Tot}=\{e:\varphi_e\text{ is total}\}$.
A binary function $f$ is an approximation to $\mathrm{Tot}$ if $\chi_{\mathrm{Tot}}=\lambda e.\lim_sf(e,s)$.
The following result is adapted from Soare \cite{Soa:1987} to the setting of Medvedev degrees.
\begin{lemma}
    \label{dftot}
    $\mathrm{DomFcn}$ is Medvedev equivalent to the mass problem of approximations to $\mathrm{Tot}$.
\end{lemma}
\begin{proof}
    We first show that given an approximation $f$ to $\mathrm{Tot}$, we can define a dominating function $h$.
    For all $e\le s$ let $$t(e)=(\mu t\ge s)[(\forall x\le s)\varphi_{e,t}(x)\downarrow \,\,\,\lor\,\,\, f(e,t)=0],$$
    and let $h(s)=\max\{t(e):e\le s\}$.
    $t(e)$ is always defined, as if $\varphi_e$ is partial then $f(e,t)=0$ for all but finitely many $t$.
    If $\varphi_e$ is total then $f(e,t)=1$ for all but finitely many $t$ and so $h(s)>\varphi_e(s)$ for all but finitely many $s$.
    \par
    Now suppose $h$ is a dominating function.
    Define $f(e,s)$ as: $$f(e,s)=\begin{cases}
        1&\text{if }(\forall z\le s)\varphi_{e,h(s)}(z)\downarrow\\
        0&\text{otherwise}
    \end{cases}$$
    If $\varphi_e$ is total then so is $\psi_e(y)=\mu s[(\forall x\le y)\varphi_{e,s}(x)\!\!\downarrow]$, so $h$ dominates $\psi_e$ and so $f(e,s)=1$ for all but finitely many $s$.
    If $\varphi_e$ is partial then $\psi_e(y)$ must diverge for some~$y$ so $f(e,s)=0$ for all $s\ge y$.
    \par
    These constructions are determined only by the given approximation to $\mathrm{Tot}$ or dominating function, so there are two fixed Turing functionals which witness the Medvedev equivalence.
\end{proof}
The following proofs are adapted from Lempp et al.\ \cite[Th.\ 4.2--4.3]{LMNS:2023}.
The proofs are presented here with more detail as to be more approachable, and to compare to similar results on ideal independent families later.
Lempp et al.\ prove the following result which is used in Theorem \ref{lmns4.2}.
The proof is omitted here as no changes would be made.
\begin{lemma}
    \label{lmns3.7}
    There is a uniformly computable sequence $P_0, P_1, \cdots$ of nonempty $\Pi^0_1$-classes such that for every $e$,
    \begin{itemize}
        \item if $\varphi_e$ is total then $P_e$ contains a single element, and
        \item if $\varphi_e$ is not total then $P_e$ contains only bi-immune elements.
            \qed
    \end{itemize}
\end{lemma}
It should be noted that if $P_e$ contains only a single element, that element must be computable.
\par
Recall the computability theoretic notion of independent families from Definition~\ref{defidp}.
\begin{theorem}
    \label{lmns4.2}
    There is a fixed Turing functional that from every maximal independent family computes a dominating function. 
    That is, $\mathrm{DomFcn}\le_s\mathcal I$.
\end{theorem}
\begin{proof}
Exactly as in Lemma \ref{lmns3.7}, let $\langle P_e\rangle_{e\in\omega}$ be a uniformly computable sequence of nonempty $\Pi^0_1$-classes.
Let $\langle Q_e\rangle_{e\in\omega}$ be the sequence of $\Pi^0_1$-classes of complements of elements of $P_e$'s, that is: $Q_e = \{X:\overline X\in P_e\}$.
For every set $C$, let ${S_C=\{X\in2^\omega:C\subseteq X\}}$ be the family of subsets of $C$.
\par
Now, let $F=\langle F_e\rangle_{e\in\omega}$ be a maximal independent family.
Recall the definition of $F_\sigma$ from (\ref{fsigma}).
We observe the following claim for each $e$,
\begin{align*}
    \varphi_e\text{ total}
    &\iff\exists\sigma\exists n(F_\sigma\smallsetminus[0,n]\subseteq X\text{ for some }X\in P_e\cup Q_e)\\
    &\iff\exists\sigma\exists n(P_e\cap S_{F_\sigma\smallsetminus[0,n]}\neq\emptyset\,\,\,\lor\,\,\,Q_e\cap S_{F_\sigma\smallsetminus[0,n]}\neq\emptyset)
\end{align*}
The latter equivalence is easy to see; we justify the former as follows.
If $\varphi_e$ is total we have $P_e = \{X\}$ for some computable set $X$.
As $F$ is maximal independent, for every $\sigma\in2^{<\omega}$ we have $F_\sigma\subseteq^*X$ or $F_\sigma\subseteq^*\overline X$.
We can certainly leave out some finite initial segment of elements from $F_\sigma$ and have it be a subset (not mod finite) so the forwards implication holds. 
If $F_\sigma\smallsetminus[0,n]$ (an infinite computable set) is a subset of some element of $P_e$, then that element must not be immune and hence not bi-immune, so $\varphi_e$ is total.
If it is instead a subset of some element of $Q_e$, then the complement of that element in $P_e$ cannot be co-immune and hence not bi-immune, so $\varphi_e$ is total.
The reverse implication holds.
\par
$S_{F_\sigma\smallsetminus[0,n]}$ is a $\Pi^0_1[F]$-class uniformly in $i$ and $n$.
The emptiness of a $\Pi^0_1[F]$ class is a $\Pi^0_1[F]$ property, so by the above equivalence we get that $\mathrm{Tot}=\{e:\varphi_e\text{ is total}\}$ is~$\Sigma^0_2[F]$.
This equivalence does not depend a single $F$, so the $\Sigma^0_2$-index for $\mathrm{Tot}$ will be fixed.
As $\mathrm{Tot}$ is also $\Pi^0_2$ it will be $\Delta^0_2[F]$ with a fixed pair of indices, so it is reducible to~$F'$ with a fixed reduction.
$\mathrm{Tot}$ is therefore limit computable from $F$ by the relativised limit lemma, and its proof shows that there is a Turing functional which computes an approximation to $\mathrm{Tot}$ given our maximal independent family.
Hence by Lemma \ref{dftot} we can uniformly compute a dominating function.
\end{proof}
\pagebreak
\begin{theorem}
    \label{lmns4.3}
    There is a fixed Turing functional that from every dominating function computes a maximal independent family.
    That is, $\mathcal I\le_s\mathrm{DomFcn}$.
\end{theorem}
It is more than sufficient to prove that a dominating function computes a family which generates freely (up to $=^*$-equivalence) the Boolean algebra of computable sets mod finite sets.
Were a finite set obtainable from finite intersections and complements of elements of this family, then one could obtain (up to $=^*$) one set in the family by finite intersections and complements of the others, and they would not freely generate the Boolean algebra.
Maximality follows as we can get any other computable set by finite intersections and complements.
\begin{proof}
Let $\langle\psi_e\rangle_{e\in\omega}$ be an effective listing of the binary valued partial computable functions which are defined only on initial segments of $\omega$.
Let ${V_{e,k}=\{x:\psi_e(x)=k\}}$ for $k\in\{0,1\}$.
Let $h\in\mathrm{DomFcn}$.
\par
In each phase $e$ we will define the computable set $F_e$, we will assume that in the previous stages we have constructed $F_i$ where $i<e$, and hence $F_\sigma$ for $|\sigma|=e$; recall the definition from (\ref{fsigma}).
We will try to build $F_e$ to be independent from the previous sets while putting in elements of $V_{e,0}$.
We define a sequence $\langle r^e_n\rangle_{e\in\omega}$ using $h$ as an oracle, though later proving that this sequence is actually computable on its own, but not necessarily with a computable index.
From each of the intervals $[r^e_n,r^e_{n+1})$ we will try to put in at least one element from $V_{e,0}$ and leave at least one out, if that is not possible then we simply try to maintain independence.
The dominating function lets us eventually choose the next term of the sequence for this to occur when it is possible.
\\\\
\emph{Construction:} We build $F_e$ as follows.
Note that we use $\sigma$ only refer to a string of length $e$.
Define $r^e_0=0$, and recursively define $r^e_{n+1}$ in terms of $r^e_n$ as the least $r>r^e_n$ such that for each $\sigma$:
\begin{enumerate}
    \item[a)] if there are $u,w\in\operatorname{dom}(\psi_{e,h(r^e_n)})\cap F_\sigma$ with $r^e_n\le u<w$ and $u\in V_{e,1},w\in V_{e,0}$ then $r>w$, and
    \item[b)] $|[r^e_n,r^e_{n+1})|\cap F_\sigma\ge2$.
\end{enumerate}
Now, for $x\in[r^e_n,r^e_{n+1})$ we define $F_e(x)$, for each $\sigma$,
\begin{itemize}
    \item if the antecedent of (a) holds and $\psi_e$ is defined on $[r^e_n,r^e_{n+1})$, then $F_e(x)=\psi_e(x)$,
    \item otherwise, if $x=\min([r^e_n,r^e_{n+1})\cap F_\sigma)$ then $F_e(x)=1$, otherwise $F_e(x)=0$.
\end{itemize}
\emph{Verification:} We now verify that this construction satisfies the desired properties.
There are several things to check.
\pagebreak
\begin{claim}
    (1) Each set $F_\sigma$ is infinite, and hence $F$ is independent.
    (2) Each sequence $\langle r^e_n\rangle_{e\in\omega}$ is infinite, and hence for each $e$ there is a total functional $\Theta_e$ that gives $F_e$.
\end{claim}
We reason by induction on $e$.
Suppose for all $\sigma$ of length $e$ that $F_\sigma$ is infinite.
We can see that $\langle r^e_n\rangle_{e\in\omega}$ is infinite as $F_\sigma$ is infinite, so there is always a next $r$ above $r^e_n$ such that $[r^e_n,r)\cap F_\sigma\ge2$, so we can satisfy (b).
Whether (a) applies or not does not affect the existence of some next $r$, only how high it is, hence the sequence is infinite.
\par
To see that $F_{\sigma\hatc i}$ is infinite for $i\in\{0,1\}$ consider the cases where $\psi_e$ is partial and total.
If $\psi_e$ is partial, then as $F_\sigma$ is infinite, infinitely often the intervals $[r^e_n,r^e_{n+1})$ contain at least one element of $F_e$ and one element of $\overline F_e$, so as $\langle r^e_n\rangle_{e\in\omega}$ is infinite, each $F_{\sigma\hatc i}$ is also infinite.
If $\psi_e$ is total, then the antecedent of (a) will be satisfied infinitely often and each $F_{\sigma\hatc i}$ will be infinite.
\begin{claim}
    Each set $F_e$ is computable.
    \label{ieref}
\end{claim}
We reason by induction $e$, so suppose for all $i<e$ that $F_i$ is computable.
Hence also each $F_\sigma$ is computable for $|\sigma|=e$.
If $\psi_e$ is partial, then for all sufficiently large $n$, the antecedent of condition (a) fails.
Hence $\langle r^e_n\rangle_{e\in\omega}$ is computable from $F_\sigma$ and hence computable, and so $F_e$ is computable.
\par
Otherwise, $\psi_e$ is total.
Define $I_e=\{\sigma:|\sigma|=e\;\land\;|F_\sigma\cap V_{e,0}|=|F_\sigma\cap V_{e,1}|=\infty\}$.
Define $p_e$ a function such that $p_e(m)$ is the least stage $s$ such that
\begin{itemize}
    \item if $\sigma\notin I_e$ then (b) holds with $r^e_n=m$ and $r=s$,
    \item if $\sigma\in I_e$, there are $u,w\in\operatorname{dom}(\psi_{e,s})$ such that $m\le u<w$ as in (a).
\end{itemize}
Let $p_e(m)=0$ if $m\neq r^e_n$ for any $n$.
Note that $p_e(m)$ is computable as $F_\sigma,V_{e,0},V_{e,1}$ are computable.
By definition $p_e(r^e_n)$ is a stage at which changes to $\psi_e$ will no longer affect~(a).
As for sufficiently large $n$ we have $h(r^e_n)\ge p(r^e_n)$, and $r$ is the minimum suitable value, it is sufficient to use this $p_e$ instead of $h$ to construct $F_e$.
Hence the individual set $F_e$ is computable.

It should be noted that the functions $p_e$ are distinct so this does not mean that $F$ is computable, it still depends on $h$ as we use it uniformly to bound all $p_e$.

\begin{claim}
    If $\psi_e$ is total, then for every string $\tau=\sigma\hatc a$ with $|\tau|=e+1$, $F_\tau\subseteq^*V_{e,0}$ or $F_\tau\cap V_{e,0}=^*\emptyset$.
\end{claim}
Define $I_e$ as in Claim \ref{ieref}.
If $\sigma\notin I_e$ then $F_\sigma\subseteq^* V_{e,i}$ for some $i$ and the claim holds. 
Otherwise, by construction in Phase $e$ we have $F_{\sigma\hatc i}\subseteq^* V_{e,i}$.
Hence $F_{\sigma\hatc0}\subseteq^*V_{e,0}$ and $F_{\sigma\hatc1}\cap V_{e,0}=^*\emptyset$, and the claim holds.
\par
This claim shows that the $=^*$-classes of the $F_e$ freely generate the $=^*$-classes of the computable sets, and in particular we have that $F$ is a maximal independent family.
\end{proof}

\subsection{The Boolean algebra of $\omega$-c.a.\ sets}
We follow the treatment of $\omega$-computably approximable sets given by Nies \cite{Nies:2008}.
For a time, these were referred to as $\omega$-c.e.\ sets; this was misleading so the term $\omega$-c.a.\ was introduced by some authors, Greenberg and Downey \cite{DG:2020} for example.
\begin{definition} \label{wca}
    A set $Z$ is called $\omega$-c.a.\ if it has a computable approximation $\langle Z_s\rangle$ such that there is some function $f$ which bounds the number of changes of each element in~$Z$. Precisely, that is
    $$f(x)\ge|\{s>x:Z_s(x)\neq Z_{s-1}(x)\}|\text{ for each }x.$$
    We refer to $f$ as the computable change bound for $Z$.
\end{definition}
This notion is of interest as the $\omega$-c.a.\ sets form a boolean algebra with unions and complements, whereas classes like the c.e.\ sets do not.
Nies \cite[Prop.\ 1.4.4]{Nies:2008} proves the following characterisation of $\omega$-c.a.\ sets.
\begin{proposition}
    $Z$ is $\omega$-c.a.\ $\iff Z\le_{wtt}\emptyset' \iff Z\le_{tt}\emptyset'$.\\
    Moreover, the equivalences are effective.
    \qed
\end{proposition}
As each truth-table reduction is determined by a (partial) computable function, this provides us with an effective listing of the $\omega$-c.a.\ sets.
A rough summary follows, though more detail is given by Nies \cite[Prop.\ 1.4.4]{Nies:2008}.
Let $\langle \psi_e\rangle_{e\in\omega}$ be an effective listing of partial computable functions which are defined on an initial segment of $\omega$
It is useful to restrict to these functions; it does not affect the presence of the total functions.
Recall the notion of strong indices for finite sets, $D_e=\{x_1<\cdots<x_n\}$ for $e=2^{x_1}+\cdots+2^{x_n}$.
According to Nies \cite[Prop.\ 1.2.21]{Nies:2008}, $Z\le_{tt}\emptyset'$ is equivalent to the existence of a computable $g$ such that the elements $x$ of $Z$ are those such that $D_{g(x)}$ (as a finite set of strings) contains an initial segment of $\emptyset'$.
In the original statement, this is written as a disjunctive normal form of a Boolean formula to fit with the definition of truth-table reductions.
This gives us the required effective listing $\langle V_e,f_e\rangle_{e\in\omega}$ by defining:
$$V_e=\{x:x\in\operatorname{dom}(\psi_e)\,\land\,(\exists\sigma\in D_{\psi_e(x)})[\sigma\preceq\emptyset']\},$$
$$f_e(x)=2\max\{|\sigma|:\sigma\in D_{\psi_e(x)}\}.$$
These $f_e$ are twice the use bounds for the functional that gives the weak truth-table reduction to $V_e$, as described by Nies \cite[Prop.\ 1.2.21]{Nies:2008}.
Conveniently, this also gives us the computable approximations:
$$V_{e,s}=\{x:x\in\operatorname{dom}(\psi_{e,s})\,\land\,(\exists\sigma\in D_{\psi_{e,s}(x)})[\sigma\preceq\emptyset'[s]]\},$$
$$f_{e,s}=2\max\{|\sigma|:\sigma\in D_{\psi_{e,s}(x)}\}.$$
It is important to note that $f_e$ converges exactly where and when $\psi_e$ does.
\par
One can generalise the definition of $\omega$-c.a.\ to higher order types, as described in depth by Downey and Greenberg \cite{DG:2020}.
A slightly different approach is taken compared to Definition \ref{wca}.
Here, $\omega$ refers to the order type and $\mathbb N$ to the underlying set.
\begin{definition}
    Let $\mathcal R=\langle\mathbb N,<_{\mathcal R}\rangle$ be a computable well-ordering of $\mathbb N$.
    A set $Z$ is called $\mathcal R$-computably approximable if there is a computable approximation $\langle Z_s\rangle$ for $Z$ and there is a computable function $g:\mathbb N^2\to\mathbb N$ such that each map $\lambda s.g(x,s)$ for fixed $x$ is non-increasing with respect to $<_{\mathcal R}$, and whenever $Z_s(x)\neq Z_{s-1}(x)$ we have ${g(x,s)\neq g(x,s-1)}$.
\end{definition}
The existence of this function $g$ with the well-foundedness of $\mathcal R$ ensures that the changes for each element are finite, and hence the limit $Z$ does exists.
This definition agrees with the above definition of $\omega$-c.a.
Intuitively, this says that changes to $Z_s(x)$ are counting down below some ordinal, giving us a more complex structure of finite changes below $\Delta^0_2$.
This is not always the same as $\omega$-c.a.\ for transfinite order types, as when $g$ `steps down' from a limit ordinal we have the additional information of the step $s$ at which this happens.
Identifying ordinals with orderings of $\mathbb N$, for ordinals $\alpha<\beta$, the class of $\alpha$-c.a.\ sets contains the class of $\beta$-c.a.\ sets, so the higher the ordinal, the looser the constraint on how the elements change.
The following theorem describes a MAD family for the Boolean algebra of $\omega$-c.a.\ sets in this way.
The construction gives an example as low as possible in this hierarchy.
It is interesting to note that by the Medvedev equivalence between MAD families and maximal towers shown by Lempp~et~al.\ \cite[Fact 2.2]{LMNS:2023} this also gives an $(\omega+1)$-c.a.\ maximal tower for the $\omega$-c.a.\ sets.

\begin{theorem}
    There is a MAD family $F$ for the $\omega$-c.a.\ sets such that $F$ is $(\omega+1)$-c.a.
\end{theorem}
\begin{proof}
    Let $\langle V_e,f_e\rangle_{e\in\omega}$ be a listing of the $\omega$-c.a.\ sets and their computable change bounds given by truth-table reductions to $\emptyset'$, as described above and by Nies \cite[Prop.\ 1.4.5]{Nies:2008}.
    We use the approximations $V_{e,s},f_{e,s}$ to these as described above.
    Define the sequence $\langle M_e\rangle_{e\in\omega}$ such that $M_{2e}=V_e$ and $M_{2e+1}=\omega$ for each $e$.
    The construction builds $\omega$-c.a.\ sets $H_e$ using the sets $M_e$, with the choice of $M_{2e+1}$ guaranteeing that $H_{2e+1}$ is always infinite.
    The family $F$ defined by $F_e=H_{2e}\cup H_{2e+1}$ will be constructed to have all desired properties, using $H_{2e+1}$ to ensure each element is infinite while maintaining almost disjointness.
    In order to verify that $F$ is maximal almost disjoint the construction will ensure the following requirements are met for every infinite $M_e$:
    $$P_e:\forall n\left[M_e\smallsetminus\bigcup_{i<n}H_i\text{ infinite}\right]\implies M_e\cap H_e\text{ infinite.}$$
    \emph{Construction:}\par
    \emph{Stage 0}: For all $e$, say that every $n>e$ is \emph{free to act on $H_e$}.\par
    \emph{Stage $s>0$:}
    We work on constructing the sets $H_e$ for $e<s$ in parallel.\par
    \emph{Substage $e<s$:}
    Look for $n$ with $e<n<s$ such that $n$ is free to act on $H_e$, and such that there is an $x\ge2n^2$ in $M_{e,s}\smallsetminus\bigcup_{i<n}H_{i,s}$.
    If found then choose $n$ least and then $x$ least for $n$.
    Put $x$ into $H_e$ and say $n$ is no longer free to act on $H_e$.
    If some $y$ was added due to the action of $m$ on $H_e$, and $y$ was removed from $M_e$ at stage $s$, then remove $y$ from $H_e$ and say $m$ is free to act on $H_e$.
    \\\\
    \emph{Verification:}
    We verify that each condition $P_e$ is met, and that all required properties of $H$ and $F$ hold.
    \begin{claim}
        Each set $H_e$ is $\omega$-c.a.
    \end{claim}
    For each $x$, $H_e(x)$ can only change to respond to a change to $M_e(x)$, and only once per change to $M_e(x)$.
    Hence, if $g$ is a computable function which bounds the number of changes of its input in $M_e$, it also bounds the number of changes of its input in $H_e$.
    Each column $H_e$ is therefore $\omega$-c.a.
    \begin{claim}
        The union $\bigcup_iH_i$ is coinfinite.
    \end{claim}
    For each $e$, each $n>e$ contributes at most one $x$ to $H_e$.
    Let $N\in\omega$ and say we have $x<2N^2$ in this union.
    It must be due to the action of some $n<N$ on some $H_e$ with $e<n$, so there are certainly less than $N^2$ many such $x$.
    Hence the union must be coinfinite.
    By the choice of each $M_{2e+1}$ this also implies that each $H_{2e+1}$ and hence each $F_e$ is infinite.
    \begin{claim}
        Each $P_e$ is satisfied.
    \end{claim}
    Suppose its hypothesis holds, so for every $n$ we have that $M_e\smallsetminus\bigcup_{i<n}H_i$ is infinite.
    For every $n$ there will eventually be some stage $s$ where we have $x\ge2n^2$ such that $x\in M_{e,s}\smallsetminus\bigcup_{i<n}H_{i,s}$ and $x$ does not leave $M_e$ past stage $s$.
    The least such $x$ that has not been included by a smaller $n$ will be added to $H_e$.
    Every $n$ will therefore act to add one new $x$ from $M_e$ into $H_e$, so $M_e\cap H_e$ is infinite, and $P_e$ is satisfied.
    \begin{claim}
        The family $H$, and hence also $F$, is almost disjoint.
    \end{claim}
    Suppose we have $e,k$ with $e<k$, we show that $H_e\cap H_k$ is finite. 
    If $x$ last enters~$H_e$ due to the action of $n$ at stage $s$, then $x\notin\bigcup_{i<n}H_{i,s}$.
    Existence of such an $s$ is guaranteed as each element can only have finitely many changes.
    As $x\in H_k$, $x$ must last enter $H_k$ at some stage $t<s$, as it can never enter $H_k$ after stage $s$.
    This implies that $n\le k$, as otherwise $x\in\bigcup_{i<n}H_{i,s}$ which contradicts $x$ entering $H_e$ at stage $s$ due to the action of $n$.
    As each $n$ can only act to include one such $x$, there can only be $k$ many such $x$, so $|H_e\cap H_k|\le k$.
    Hence $H$ is almost disjoint, and $F$ must also be almost disjoint.
    \begin{claim}
        $F$ is maximal almost disjoint.
    \end{claim}
    It is sufficient that for every infinite $M_e$ we find $k$ such that $M_e\cap F_k$ is infinite.
    As every $P_e$ is satisfied, if its hypothesis holds then $M_e\cap H_e$ and hence $M_e\cap F_{\lfloor e/2\rfloor}$ is infinite.
    If its hypothesis does not hold, then $M_e\subseteq^*\bigcup_{i<n}H_i$ for some $n$, and hence there is some $i<n$ such that $M_e\cap H_i$ is infinite.
    \begin{claim}
        $F$ is $(\omega+1)$-c.a.
    \end{claim}
    For this claim we will refer to $\omega$ as the order type (and element of $\omega+1$) and $\mathbb N$ as the underlying set.
    Identify $\omega+1$ with a computable well ordering of $\mathbb N$ with the same order type.
    Define $g:\mathbb N^2\to\omega+1$ as follows, on each pair $\langle e,x\rangle$.
    Whenever the first input is not of the form $\langle e,x\rangle$, output $0$.
    Define $g(\langle e,x\rangle,s)=\omega$ where $f_{e,s}(x)$ is undefined (note that before $f_{e,s}(x)$ converges no elements may enter $V_e$).
    If $s$ is the first stage such that $f_{e,s}(x)$ converges, then define $g(\langle e,x\rangle,s)=f_{e,s}(x)+1$.
    This is the sum of the change bounds for $H_{2e}$ and $H_{2e+1}$.
    For all stages $t>s$, if $F_{e,t}(x)\neq F_{e,t-1}(x)$ then set $g(\langle e,x\rangle,t)=g(\langle e,x\rangle,t-1)-1$.
    Otherwise set $g(\langle e,x\rangle,t)=g(\langle e,x\rangle,t-1)$.
    Then $g$ is a non-increasing function into $\omega+1$, and when $F_{e,s}(x)\neq F_{e,s-1}(x)$ we have $g(\langle e,x\rangle, s)\neq g(\langle e,x\rangle, s-1)$, hence $F$ is $(\omega+1)$-c.a.
\end{proof}

\section{Ideal Independent Families} \label{iif}
The cardinal characteristic $\smm$ of maximal ideal independent families was introduced by Monk \cite{Mon:2008}.
A family $\mathcal A\subseteq\infsets$ is said to be \emph{ideal independent} if for every finite subfamily $\mathcal X\subseteq\mathcal A$ and $A\in\mathcal A\smallsetminus\mathcal X$, the set $A\smallsetminus\bigcup\mathcal X$ is infinite.
Monk studied this and other cardinal characteristics in the setting of Boolean algebras.
Given a Boolean algebra of sets $\mathbb B$, he defines:
$$\smm(\mathbb B)=\min\{|\mathcal A|:\mathcal A\subseteq\mathbb B,\mathcal A\text{ is maximal ideal independent}\}.$$
The letter $\mathfrak s$ is used due to similarities with the cardinal invariant of spread described by Monk \cite{Mon:1996}, not because of any relation to the splitting number.
It was originally written with a plain font as $s_{\mathrm{mm}}$, though that nuance seems to have been lost in more recent work.
It is simply referred to as $\smm$ in the case where the Boolean algebra is subsets of $\omega$ modulo finite sets.
\par
The characteristic $\smm$ was studied in the setting of set theory by Monk \cite{Mon:2008,Mon:2012}; by Cancino, Guzm\'an, and Miller \cite{CGM:2021}; and by Bardyla et al.\ \cite{BCFS}.
Their work includes proofs of relations such as $\mathfrak d,\mathfrak u\le\smm$, and that $\smm$ and $\mathfrak i$ are independent.
It is still an open question whether or not $\smm<\mathfrak a$ is consistent with ZFC, as is the case with $\mathfrak i$.
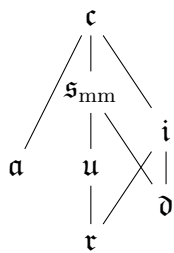
\begin{figure}[h]
\begin{tikzpicture}
    \node (u) {$\mathfrak u$};
    \node (smm) [above of = u] {$\smm$};
    \node (c) [above of = smm] {$\mathfrak c$};
    \node (i) [right of = u,yshift=0.5cm] {$\mathfrak i$};
    \node (r) [below of = u] {$\mathfrak r$};
    \node (d) [right of = r,yshift=0.5cm] {$\mathfrak d$};
    \node (a) [left of = u] {$\mathfrak a$};

    \draw (u) -- (r);\draw (c) -- (smm);\draw (c) -- (a);
    \draw (c) -- (i);\draw (i) -- (r);\draw (i) -- (d);
    \draw (smm) -- (u); \draw (smm) -- (d);
\end{tikzpicture}
\centering
\caption{Placement of $\smm$ among other cardinal characteristics}
\label{smmdiag}
\end{figure}
As Lempp et al.\ \cite{LMNS:2023} did for the independence and ultrafilter numbers, we define a mass problem for the maximal ideal independent families in a Boolean algebra.
Let~$\mathbb B$ be a Boolean algebra of sets modulo finite sets.
It is useful to use the following characterisation of maximal ideal independent families.
A sequence $\langle F_e\rangle_{e\in\omega}$ in $\mathbb B$ is maximal ideal independent if it is ideal independent, and for any infinite and coinfinite set $R\in\mathbb B$, either there is an $n$ such that $R\subseteq^*\bigcup_{i<n}F_i$, or there is some $e$ and finite set $B\subseteq\omega$ such that $F_e\subseteq^*R\cup\bigcup_{i\in B}F_i$.
It should be noted that every almost disjoint or independent family is also ideal independent, but MAD families and maximal independent families are not necessarily maximal ideal independent.
We denote by $\mathcal I^{id}$ the mass problem of maximal ideal independent families for computable sets modulo finite sets.
This notation is chosen to represent the similar behaviour it has to the mass problem of (full) independent sets.
\par
The main results of this section show that $\mathcal I^{id}\equiv_s\mathrm{DomFcn}$, and hence $\mathcal I^{id}$ joins the equivalence with $\mathcal I$ and $\mathcal U$, as per Lempp et al.\ \cite{LMNS:2023}.
The proofs follow similar approaches to those used by Lempp et al.\ for independent sets and ultrafilter bases.
These results show that even more of Figure \ref{smmdiag} collapses in the setting of computability.
We also show that there is a $\Delta^0_2$ maximal ideal independent family for the $K$-trivials, another Boolean algebra of sets that sits just barely above the computable sets.

\subsection{Ideal independent families and dominating functions}
\begin{theorem}
    \label{dfbelowiid}
    There is a fixed Turing functional that from every maximal ideal independent family computes a dominating function. 
    That is, $\mathrm{DomFcn}\le_s\mathcal I^{id}$.
\end{theorem}
\begin{proof}
We use the setting from Theorem \ref{lmns4.2}. 
Let $F=\langle F_e\rangle_{e\in\omega}$ be a maximal ideal independent family. 
Define the sets $F_\sigma$ as in (\ref{fsigma}).
Let $\langle P_e\rangle_{e\in\omega}$ and $\langle Q_e\rangle_{e\in\omega}$ be sequences of $\Pi^0_1$-classes as in Lemma \ref{lmns3.7}.
As $F$ is maximal ideal independent, for any infinite, co-infinite, computable $R$ one of the following holds.
\begin{itemize}
    \item[(1)] For some $n$ we have $R\subseteq^*\bigcup_{i<n}F_i$.\\
        This can also be written as $F_\sigma=\bigcap_{i<n}\overline F_i\subseteq^*\overline R$, where $\sigma$ has 0 for every bit.
    \item[(2)] There is an $e$ and finite set of indices $B$ such that $F_e\subseteq^*R\cup\bigcup_{i\in B}F_i$.\\
        Similarly, this can be written as $F_\sigma=\overline F_e\cap\bigcap_{i\in B}F_i\subseteq^*R$, where $\sigma$ has 1 for at most one bit.
\end{itemize}
Now we restrict $\sigma$ to only have the forms described above, and as in Theorem \ref{lmns4.2} we get exactly as before
\begin{align*}
    \varphi_e\text{ total}
    &\iff\exists\sigma\exists n(F_\sigma\smallsetminus[0,n]\subseteq X\text{ for some }X\in P_e\cup Q_e)\\
    &\iff\exists\sigma\exists n(P_e\cap S_{F_\sigma\smallsetminus[0,n]}\neq\emptyset\,\,\,\lor\,\,\,Q_e\cap S_{F_\sigma\smallsetminus[0,n]}\neq\emptyset)
\end{align*}
Again we inspect the first equivalence.
By the conditions on computable sets above, the forward implication follows just as in Theorem \ref{lmns4.2}. 
The reverse implication is now weaker, a sufficient $\sigma$ which meets our restrictions also works just fine in the general case, so the reverse implication holds.
As we have already seen this is sufficient to provide a Turing functional which computes a dominating function from our maximal ideal independent family.
\end{proof}
\pagebreak
\begin{theorem}
    \label{iidbelowdf}
    There is a fixed Turing functional that from every dominating function computes a maximal ideal independent family.
    That is, $\mathcal I^{id}\le_s\mathrm{DomFcn}$.
\end{theorem}
\begin{proof}
The setting is the same as in Theorem \ref{lmns4.3}.
Let $\langle\psi_e\rangle_{e\in\omega}$ be an effective listing of the binary valued partial computable functions which are defined on an initial segment of $\omega$.
Let $V_{e,k}=\{x:\psi_e(x)=k\}$ for $k\in\{0,1\}$.
Let $h\in\mathrm{DomFcn}$.
\par
We define $F$ inductively in $e$, so each $F_e$ is constructed in terms of the $F_i$ for $i<e$ which have already been constructed.
We use the notation $F_{<e}=\bigcup_{i<e}F_i$, and in constructing $F$ we ensure that for all $e$ we have $F_e\cap F_{<e}=\emptyset$, so $F$ is pairwise disjoint and hence ideal independent.
We observe the sufficient conditions for a maximally ideal independent family used in Theorem \ref{dfbelowiid}.
As we build $F_e$, if $V_{e,0}$ appears to lie in $F_{<e}$ then we just need to maintain independence as the first condition holds.
Otherwise we follow $V_{e,0}$ so that we satisfy the second condition, introducing a new dependence when $V_{e,0}$ is added.
\\\\
\emph{Construction:} To build $F_e$, we assume that for $i<e$ we already have $F_i$.
These~$F_i$ are disjoint, computable, and their union $F_{<e}$ is coinfinite.
We use $h$ to define a sequence~$\langle r^e_n\rangle_{n\in\omega}$.
Set $r^e_0=0$. Define $r^e_{n+1}$ in terms of $r^e_n$ to be the least $r$ such that:
\begin{align*}
    \text{(a)  } & \text{if there are $x,y$ with $r^e_n\le x<y$ and $x,y\in \operatorname{dom}(\psi_{e,h(r^e_n)})\smallsetminus F_{<e}$ then $r>y+1$}\\
    \text{(b)  } & |[r^e_n,r)\smallsetminus F_{<e}|\ge2
\end{align*}
Note that (a) implies (b) when the assumption in (a) holds.
If the assumption of (a) holds, then we put $x$ into $F_e$.
Otherwise we put $\min([r^e_n,r^e_{n+1})\smallsetminus F_{<e})$ into $F_e$.
At each step here we put in one element and leave one out, all outside of $F_{<e}$.
As $\langle r^e_n\rangle_{n\in\omega}$ is infinite, we then have that $F_e$ is infinite, and $ F_{<e+1}$ is coinfinite.
\par
This process is determined only by $h$ so there is a Turing functional $\Theta_e$, determined uniformly in $e$, such that $F_e=\Theta^h_e$.
\\\\
\emph{Verification:} We verify this construction for $F$ has the desired properies. As all elements put into $F_e$ are in the complement of $F_{<e}$, $F$ is pairwise disjoint and hence ideal independent.
\par
\begin{claim}
    Each set $F_e$ is computable.
\end{claim}
We define $p_e(n)$ to be the least stage where $V_{e,0}$ `sees' $x$ and $y$.
More formally, $p_e(n)$ is the least $s$ where $[n,s)$ witnesses (b) and there are $x,y\in V_{e,0,s}$ such that $n\le x<y$ that are suitable for $(a)$ (if $n$ is not of the form $r^e_n$ then just let $p_e(n)=0$).
This is computable, so as $h$ eventually dominates $p_e$, for some $n$ we will be defining $r^e_{n+1}$ by checking $V_{e,0}$ at some stage $h(r_n)\ge p_e(r^e_n)$.
We chose $r^e_{n+1}$ to be minimal so we might as well check at stage $p_e(r^e_n)$, so $\langle r^e_n\rangle_{n\in\omega}$ is computable, and hence so is $F_e$.
\begin{claim}
    $F$ is a maximal ideal independent family.
\end{claim}
It is sufficient for maximality that whenever $V_{e,0}\cup V_{e,1}=\omega$ (i.e.\ $\psi_e$ is total) then $V_{e,0}\not\subseteq^* F_{<e}$ implies that $F_e\subseteq^*V_{e,0}$.
If $V_{e,0}$ is not coinfinite then the consequent is always true, and if it is not infinite then the antecedent is always false, so henceforth we assume $V_{e,0}$ is infinite and co-infinite.
Suppose $V_{e,0}\not\subseteq^* F_{<e}$, so $V_{e,0}\smallsetminus F_{<e}$ is infinite.
Hence for any $r^e_n$ there will by a stage $p_e(r^e_n)$ that is sufficient to check for the antecedent in (a) to hold.
Then $h(r^e_n)>p_e(r^e_n)$ for sufficiently large $n$, so there are only finitely many $n$ where the assumption in (a) fails and elements outside of $V_{e,0}$ are added to $F_e$.
Hence $F_e\subseteq^*V_{e,0}$.
\end{proof}

\subsection{The Boolean algebra of $K$-trivial sets}
Again, we follow the treatment of Nies \cite{Nies:2008}.
A Turing machine $M$ is called \emph{prefix-free} if its domain is prefix-free, so for any strings $\sigma,\rho$ in the domain of $M$, we have $\sigma\preceq\rho\rightarrow\sigma=\rho$.
The shortest input to a prefix-free machine $M$ that gives an output $x$ is 
$$K_M(x)=\min\{|\sigma|:M(\sigma)=x\}.$$
$K$ is used to distinguish from $C$ where the machine is not required to be prefix-free.
A prefix-free machine $R$ is called \emph{optimal} if for every machine $M$
$$\forall x[K_R(x)\le K_M(x)+d_M]$$
for some constant $d_M$.
We fix an arbitrary optimal prefix-free machine $R$, then abbreviate $K(x)=K_R(x)$ as the \emph{prefix-free complexity} of $x$.
\par
A set $A$ is called \emph{$K$-trivial} if there is some fixed constant $b$ such that $$\forall n[K(A\restriction n)\le K(n)+b].$$
Here, we identify finite sets with their strong indices.
This says that the complexity of initial segments of $A$ grows about as slowly as possible, so the $K$-trivial sets lie very close to the computable sets.
\par
The $K$-trivial sets form a Boolean algebra so we can study the complexity of certain families of them as we have done for the computable sets.
Lempp et al.\ \cite[Th.\ 6.1]{LMNS:2023} show that there is a $\Delta^0_2$ ultrafilter base for the $K$-trivial sets.
The following result adapts this to the setting of maximal ideal independent families.
\begin{theorem}
    \label{d02miiktriv}
    There is a maximal ideal independent family $F$ for the Boolean algebra of the $K$-trivials such that $F$ is $\Delta^0_2$.
\end{theorem}
\begin{proof}
    Let $h$ be a function which dominates all $K$-trivial computable functions, as noted to exist by Ku\v cera and Slaman \cite{KS:2009}.
    Let $\langle V_{e,0},V_{e,1}\rangle$ be a uniform listing of the $K$-trivials and their complements given by wtt-reductions to $\emptyset'$.
    Such a listing is shown to exist by Downey et al.\ \cite{DHNS:2002}, their listing includes the constants $b$ but these are not necessary here.
    Let $T=\{0,1\}^{<\infty}$. 
    \par 
    Given any set $A=\{a_i:i\in I\}$ with $i<j\rightarrow a_i<a_j$ for any initial segment~$I$ of~$\omega$, we define $\ell(A)$ to be the set of elements of the form $a_{2i}$.
    For each $\alpha\in T$ of length $|\alpha|=e$ we define the set $S_\alpha$ to be one attempt at constructing $F_e$ for our maximal ideal independent family. 
    $S_\alpha$ will be $K$-trivial but not necessarily infinite.
    Of course, finite sets $S_\alpha$ will not be used in the constructed sequence $F$.
    Define $S_{\preceq\alpha}=\bigcup_{\beta\preceq\alpha}S_\beta$.
    We define $S_\emptyset=\ell(\omega)$ and $S_{\alpha\hatc k}=\ell(V_{e,k}\smallsetminus S_{\preceq\alpha})$ with $e=|\alpha|$, $k\in\{0,1\}$.
    This definition allows us to maintain disjointness from previous stages, and leave out infinitely many elements for future stages to fit into.
    \par
    We can see inductively that there is a path $g\in\{0,1\}^\omega$ such that $S_{g\restriction e}$ is infinite for all $e\in\omega$.
    $S_\emptyset$ is infinite, so let $e\in\omega$ and suppose there is $\alpha$ of length $e$ with $S_\alpha$ infinite, and such that $S_{\preceq\alpha}$ is coinfinite.
    If $V_{e,0}\not\subseteq^*S_{\preceq\alpha}$ then $S_{\alpha\hatc0}$ is coinfinite.
    If $V_{e,0}\subseteq^*S_{\preceq\alpha}$ then $V_{e,1}\smallsetminus S_{\preceq\alpha}$ is infinite as $S_{\preceq\alpha}$ is coinfinite, so $S_{\alpha\hatc1}$ is infinite.
    Take $S_{\alpha\hatc k}$ to be the infinite set, with preference given to the left path.
    Now we show that the next union is still coinfinite.
    $S_{\alpha\hatc k}\cup S_{\preceq\alpha}=\ell(V_{e,k}\smallsetminus S_{\preceq\alpha})\cup S_{\preceq\alpha}$ which is coinfinite as $S_{\preceq\alpha}$ is coinfinite and $V_{e,k}$ contains infinitely many elements of $S_{\preceq\alpha}$, of which infinitely many will be left out.
    \\\\
    \emph{Construction:}
    $F_e$ will be defined as $\{a^e_0,a^e_1,\cdots\}$ for an increasing sequence $\langle a^e_k\rangle_{k\in\omega}$ defined as follows.
    Define $a^e_0=0$. Given $a^e_{k-1}$ already defined we attempt to set $\alpha$ the leftmost possible with length $e$ such that there are at least $k+1$ elements in $S_\alpha$ that are less than $h(k)$.
    If such $\alpha$ exists then let $a^e_k$ be the $k$-th element of $S_\alpha$ or $a^e_{k-1}$, whichever is greater.
    If there is no such $\alpha$ let $a^e_k=a^e_{k-1}$.
    This process is computable in~$\emptyset'$, uniformly in $e$.
    \\\\
    \emph{Verification:}
    Let $g\in\{0,1\}^\omega$ be the leftmost path such that $S_{g\restriction e}$ is infinite for all~$e$.
    Fix $e$ and let $\alpha = g\restriction e$.
    Let $p(k)$ be the $(k+1)$-st element of $S_\alpha$.
    This is computable in a $K$-trivial set so it is dominated by $h$ so $F_e$ will eventually always pick $\alpha$ to pull its elements from, so $F_e=^*S_\alpha$.
    This confirms that each $F_e$ is $K$-trivial.
    \par
    By construction each $S_\alpha$ is disjoint from the preceeding $S_\beta$ sets, so $F$ is almost disjoint and hence ideal independent.
    Given an infinite $K$-trivial set $R$, let $e$ be an index such that $V_{e,0}=R$, let $\alpha=g\restriction(e-1)$.
    Suppose $R=V_{e,0}\not\subseteq^*F_{<e}=^*S_{\preceq\alpha}$ then $S_{\alpha\hatc0}=\ell(V_{e,0}\smallsetminus S_{\preceq\alpha})$ is infinite.
    Hence $g(e)=0$ and $F_e=^*S_{\alpha\hatc0}\subseteq V_{e,0}=R$.
    Using the conditions from Theorem \ref{dfbelowiid} we see that if (1) does not hold, then (2) must hold, and so $F$ is maximal ideal independent.
\end{proof}

\section{Conclusions}
The results of Lempp et al.\ \cite{LMNS:2023} and those in Section \ref{iif} show that the mass problem of maximal ideal independent families coincides with that of maximal independent families and ultrafilter bases in the Medvedev lattice.
The diagram adapted from Blass' survey \cite{Bla:2010} in Figure \ref{hasseccc} shows how rare it is for two cardinal characteristics to always be equal in ZFC.
Perhaps unfortunately, in the framework used by Lempp et al.\ for finding analogues of characteristics relating to families of sets, it seems that the separation between characteristics is often lost.
The equivalences shown in their paper and expanded on in Section \ref{iif} above give the following diagram:
\begin{figure}[H]
\begin{minipage}{.2\textwidth}
\begin{tikzpicture}
    \node (u) {$\mathfrak u$};
    \node (smm) [above of = u] {$\smm$};
    \node (i) [right of = smm] {$\mathfrak i$};
    \node (d) [right of = u] {$\mathfrak d$};
    \node (a) [left of = u] {$\mathfrak a$};
    \node (t) [below of = u] {$\mathfrak t$};

    \draw (smm) -- (u); \draw (smm) -- (d); \draw (i) -- (d);
    \draw (a) -- (t); \draw (u) -- (t); \draw (d) -- (t);
\end{tikzpicture}
\end{minipage}
\begin{minipage}{.2\textwidth}
\begin{tikzpicture}
    \node (uc) {$\mathcal U$};
    \node (smmc) [above of = uc] {$\mathcal I^{id}$};
    \node (ic) [right of = smmc] {$\mathcal I$};
    \node (dc) [right of = uc, text width=0.5cm] {$\mathrm{DomFcn}$};
    \node (tc) [below of = uc] {$\mathcal T$};
    \node (ac) [left of = tc] {$\mathcal A$};

    \draw[double] (smmc) -- (uc);\draw[double] (smmc) -- (ic);\draw[double] (dc) -- (ic);
    \draw[double] (dc) -- (uc);\draw[double] (ac) -- (tc);\draw[->] (uc) -- (tc);
    \draw[->] (uc) -- (ac);
\end{tikzpicture}
\end{minipage}
\centering
\caption{The collapse of relations between several analogues of cardinal characteristics}
\label{hasseccc2}
\end{figure}
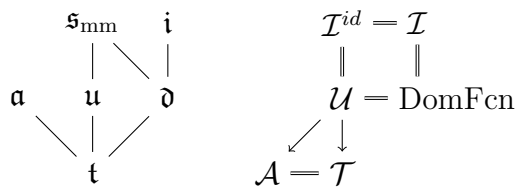
The original work on analogues of Cicho\'n's diagram by Rupprecht \cite{Rup:2010-1} and then Brendle et al.\ \cite{BBNN:2015} seems to only have this problem to a lesser extent, still retaining separation between many of the analogues.
This could suggest that the framework for analogues characteristics for families of sets is less effective in capturing the structure in the set theoretic setting or that there is some sense in which characteristics such as $\mathfrak u$ and $\mathfrak i$ are closer than independent characteristics in Cicho\'n's diagram.
\par
Other Boolean algebras, such as the $\omega$-c.a.\ and $K$-trivial sets mentioned in this dissertation, could prove interesting to look at in more depth.
The results discussed here regarding these concern only the complexities of certain families for them, but the relations between their respective mass problems could prove interesting.
Lempp et al.\ prove the equivalence between $\mathcal A$ and $\mathcal T$ for an arbitrary Boolean algebra, though most of their results are for the computable sets specifically.
It could be interesting to see which results hold for different Boolean algebras (or for arbitrary Boolean algebras) and how they relate to cardinal invariants on Boolean algebras as studied by Monk.
\par
There are still many cardinal characteristics that have not yet had analogues studied in the setting of computability, those mentioned in Blass' survey \cite{Bla:2010} but also in various other places in the literature, as was the case with $\smm$.
Also within the study of cardinal characteristics in set theory are higher cardinal characteristics between $\kappa^+$~and~$2^\kappa$ for $\kappa>\omega$.
Analogues of Cicho\'n's diagram for generalised Baire and Cantor spaces have been studied in this way by Baumhauer, Goldstern, and Shelah \cite{BGS:2020}; Brendle et al.\ \cite{BBFM:2018}; Brendle \cite{Bre:2022}; and Switzer \cite{Swi:2020-1}.
Their work provides a similarly rich structure as the case $\kappa=\omega$.
The rich structure of similar characteristics for higher cardinals could suggest that analogies in higher generalisations of computability may prove interesting alongside the computability case.

\addcontentsline{toc}{section}{References}
\printbibliography

\end{document}